\documentclass[11pt,oneside,leqno]{amsart}
\allowdisplaybreaks
\usepackage[margin=2.5cm]{geometry}
\usepackage{multirow,tabularx}
\newcolumntype{Y}{>{\centering\arraybackslash}X}

\usepackage{pbox}
\usepackage{parskip}
\usepackage{eucal}
\usepackage{amscd}
\usepackage{amssymb}
\usepackage{amsmath}
\usepackage{amsfonts}
\usepackage{amsopn}
\usepackage{latexsym}
\usepackage{hyperref}
\usepackage[nodisplayskipstretch]{setspace}
\usepackage{enumerate}
\usepackage[utf8]{inputenc}

\makeatletter
\@namedef{subjclassname@2010}{%
  \textup{2010} Mathematics Subject Classification}
\makeatother

\newtheorem{thm}{Theorem}[section]
\newtheorem{cor}[thm]{Corollary}
\newtheorem{lem}[thm]{Lemma}
\newtheorem{prop}[thm]{Proposition}

\newtheorem{rmk}[thm]{Remark}

\makeatletter
\newcommand{\BIG}{\bBigg@{2}}
\newcommand{\vast}{\bBigg@{3}}
\newcommand{\Vast}{\bBigg@{5}}
\makeatother

\numberwithin{equation}{section}

\begin{document}
\setlength{\arrayrulewidth}{0.1mm}



\title[On a Conjecture of Mordell]{On a Conjecture of Mordell}

\author[Debopam Chakraborty]{Debopam Chakraborty}
\address{Department of Mathematics\\ BITS-Pilani, Hyderabad campus\\
Hyderabad, INDIA}
\email{debopam112358@gmail.com}

\author[Anupam Saikia]{Anupam Saikia}
\address{Department of Mathematics\\ Indian Institute of Technology, Guwahati\\
Guwahati-781039, Assam, INDIA}
\email{a.saikia@iitg.ac.in}

\date{}

\begin{abstract}
A conjecture of Mordell in \cite{Mor} states that if $p$ is a prime and $p$ is congruent to $3$ mod $4$, then $p$ does not divide $y$ where $(x,y)$ is the fundamental solution to $x^{2}-py^{2}=1$. The conjecture has been verified in \cite{Beach} for primes not exceeding $10^{7}$. In this article, we show that Mordell's conjecture holds for four conjecturally infinite families of primes.
\end{abstract}


\maketitle

\section{Introduction}

There is a famous conjecture of Ankeny, Artin and Chowla in \cite{AAC} concerning the fundamental unit of the real quadratic field $\mathbb{Q}(\sqrt{p})$ where $p$ is a prime congruent to $1$ mod $4$. Their conjecture states that if $\dfrac{x+y\sqrt{p}}{2}$ is the fundamental unit of $\mathbb{Q}(\sqrt{p})$ then $y$ is not divisible by $p$. Mordell's conjecture is very similar in nature to that of Ankeny, Artin and Chowla. Mordell's conjecture states that if $x+y\sqrt{p}$ is the fundamental unit of $\mathbb{Q}(\sqrt{p})$ for a prime $p$ congruent to $3$ modulo $4$, then $p$ does not divide $y$ (\cite{Mor}). In other words, it predicts that $p$ does not divide $y$ where $(x,y)$ is the fundamental solution to $x^{2}-py^{2}=1$ when $p\equiv 3$ (mod $4$).

The conjecture of Ankeny, Artin  and Chowla (A-A-C conjecture) was first verified up to primes $p<10^{11}$ by Van Der Poorten et al. in \cite{Poorten}. In a corrigenda to \cite{Poorten}, the authors reported verification of the A-A-C conjecture for primes up to $2.10^{11}$ in \cite{Poorten2}. In \cite{Mor1}, Modell proved the A-A-C conjecture for any regular prime $p$, i.e., when $p$ does not divide the class number of the number field $\mathbb{Q}(e^{\frac{2\pi i}{p}})$. It has been conjectured that there are infinitely many regular primes, in fact Siegel conjectured that nearly $60.65\%$ of primes are regular. Thus the A-A-C conjecture holds for the conjecturally infinite family of regular primes.

The conjecture of Mordell has also been verified for primes not exceeding $10^{7}$ in \cite{Beach}. In this article, we provide an equivalent criterion (Theorem \ref{mord}) for non-divisibility of $y$ by $p$. As a consequence, we show that Mordell's conjecture holds when the regular continued fraction (RCF) of $\sqrt{p}$ has period length $2$, $4$, $6$ or $8$ (Corollaries \ref{mord2} and \ref{mord8}). By a conjecture of P. Chowla and S. Chowla in \cite{CC}, there exist infinitely many primes $p$ such that the RCF of $\sqrt{p}$ has period length $k$ for any natural number $k$. Therefore, our work provides four conjecturally infinite families of primes for which Mordell's Conjecture holds. The existence of infinitely many primes $p$ such that the RCF of $\sqrt{p}$ has period length $2$, $4$ or $6$ also follows from an old conjecture due to Bouniakowsky (\cite{Bou}) as discussed in the final section of this article. It may be noted that Hashimoto (\cite{Hash}) proved the A-A-C conjecture when the RCF of $\sqrt{p}$ has period length $1$, $3$ or $5$. As a consequence of our approach, one can easily deduce that the central term in the RCF of $\sqrt{p}$ for $p\equiv 3$ mod $4$ is either $\lfloor\sqrt{p}\rfloor$ or $\lfloor\sqrt{p}\rfloor-1$, whichever is odd (Corollary \ref{central}). We further prove a sharper bound (Proposition \ref{per6}) for the denominator of the second convergent in the RCF of $\sqrt{p}$ that we need in our work.

\section{The convergents of $\sqrt{p}$}

 When $p$ is a prime, we know that the regular continued fraction (RCF) of $\sqrt{p}$ is
of the form (e.g., see \cite{D})
$$\sqrt{p}=n + \cfrac{1}{a_1 +\frac{1}{a_{2}+ \cdots +\frac{1}{a_{l-1}+\frac{1}{2n+\frac{1}{a_{1}+\frac{1}{a_{2}+\cdots}}}}}}, \;\;\;
\mbox{ where }n = \lfloor\sqrt{p}\rfloor, \;\;\mbox{ and }a_{i} =a_{l-i}.$$
We denote it as $\sqrt{p}=\langle{n,\overline{a_{1},a_{2},\ldots ,a_{l-1},2n}}\rangle.$
Here, the first $l-1$ terms $a_{1},a_{2},\ldots,a_{l-1}$ of the period $(a_{1},a_{2},\ldots ,a_{l-1},2n)$ form a palindrome.
We establish a deeper relation between the continued fraction of $\sqrt{p}$ and the fundamental unit of the real quadratic field $\mathbb{Q}(\sqrt{p})$ that yields our results.

We first establish a relation that hold for the convergents of the continued fraction of $\sqrt{p}$ for any prime $p$. Let
$$\sqrt{p}=\langle{n,\overline{a_{1},a_{2},\ldots,a_{l-1},2n}\rangle}.$$
The $i$-th convergent of the continued fraction of $\sqrt{p}$ is given by
\begin{equation}
\label{cgnt}
\frac{k_{i}}{h_{i}}=n+\frac{1}{a_1 +\frac{1}{a_{2} +\frac{1}{\cdots+\frac{1}{a_{i}}}}}.
\end{equation}
We can write the first few convergents as
\begin{equation}
\begin{split}
h_{0}&=1, \;\;\;\;k_{0}=n;\qquad \;\;\;\; h_{1}=a_{1},\;\;k_{1}=na_{1}+1;\\
\qquad h_{2}& =1+a_{1}a_{2}, \;\;\qquad \quad \;\;\;\;k_{2}=na_{1}a_{2}+n+a_{2}.
\end{split}
\end{equation}
By convention, we take $$h_{-1}=0, \qquad k_{-1}=1.$$
The following recurrence relations satisfied by $k_{i}$ and $h_{i}$ are easy to verify.
\begin{equation}
 \label{rr}
h_{i+1}=a_{i+1}h_{i}+h_{i-1}, \qquad k_{i+1}= a_{i+1}k_{i}+k_{i-1}\;\; \mbox{  for  }\;\;i\geq 2.
\end{equation}
It can be readily verified that
\begin{equation}
\label{hk}
k_{l-1}=nh_{l-1}+h_{l-2}.
\end{equation}
The following relations involving the convergents are well-known and can be easily proved (see \cite{D}):
\begin{equation}
\label{rec1}
\begin{split}
k_{i}h_{i-1} - k_{i-1}h_{i} = (-1)^{i-1}\qquad \mbox{  for  }{i\geq 0},\\
\big(nh_{l-1}+h_{l-2}\big)^{2}-ph_{l-1}^{2}=(-1)^{l-2}.
\end{split}
\end{equation}
It follows from the second relation above that if $p$ is congruent to $3$ modulo $4$ then the length $l$ of the period of $\sqrt{p}$ has to be even as otherwise, $-1$ would be a quadratic residue of $p$. In particular, the period of $\sqrt{p}$ for any prime congruent to $3$ modulo $4$ has to be even.

\noindent We establish a relation for convergents that we use later.
\begin{prop}
\label{rec2}
Let $l$ be the length of the period of the RCF of $\sqrt{p}$. Then $$h_{l-1} = h_{i}h_{l-1-i} + h_{i-1}h_{l-2-i} \mbox{    for   } 0 \leq i \leq l-2. $$
\end{prop}
\begin{proof}
We proceed by induction. The statement is true for $i=0$ because $h_{0}=1$ and $h_{-1}=0$. Assume that it is true for $i$. By induction hypothesis and (\ref{rr}),
\begin{align*}
h_{l-1} & =  h_{i}(a_{l-1-i}h_{l-2-i}+h_{l-3-i})+h_{i-1}h_{l-2-i} \\
 & = h_{i}(a_{i+1}h_{l-2-i}+h_{l-3-i})+h_{i-1}h_{l-2-i} \qquad (\mbox{as }a_{l-1-i}=a_{i+1})\\
 & =  (a_{i+1}h_{i}+h_{i-1})h_{l-2-i} + h_{i}h_{l-3-i} \\
& = h_{i+1}h_{l-1-(i+1)}+h_{(i+1)-1}h_{l-2-(i+1)},
\end{align*}
showing that the relation holds for $i+1$ as long as $l-2-(i+1)\geq -1$, i.e., $i\leq l-2$.
\end{proof}
\noindent Later, we need the following special case of the above proposition obtained by putting $i=\frac{l}{2}-1$ where the length $l$ of the period length of the RCF of $\sqrt{p}$ is known to be even.
\begin{equation}
\label{id1}
h_{l-1} = h_{\frac{l}{2}-1}h_{\frac{l}{2}} + h_{\frac{l}{2}-2}h_{\frac{l}{2}-1}=h_{\frac{l}{2}-1}\Big(h_{\frac{l}{2}} + h_{\frac{l}{2}-2}\Big)=h_{\frac{l}{2}-1}c_{\frac{l}{2}-1},
\end{equation}
where we write $c_{\frac{l}{2}-1}=h_{\frac{l}{2}} + h_{\frac{l}{2}-2}$ for brevity.
\noindent Note that the relation is valid for $l=2$, with the convention that $h_{-1}=0$.


\section{The fundamental unit of $\mathbb{Q}(\sqrt{p})$}
It is well known by Dirichlet's theorem that the units in the ring of integers of a real quadratic field form an abelian group of rank one, and the smallest unit $>1$ is referred to as the fundamental unit. Let
$$\eta_{p}=x+y\sqrt{p}\in \mathbb{Z}[\sqrt{p}]$$ denote the fundamental unit of the real quadratic field $K=\mathbb{Q}(\sqrt{p})$, where $p$ is a prime congruent to $3$ modulo $4$ as before. The fundamental unit $\eta_{p}$ is intimately connected with the convergents of $\sqrt{p}$. It is well-known that
\begin{equation}
\label{funda1}
\eta_{p}= k_{l-1}+h_{l-1}\sqrt{p}=nh_{l-1}+h_{l-2}+h_{l-1}\sqrt{p}.
\end{equation}

When $p\equiv 3$ mod $4$, $-1$ is not a quadratic residue of $p$, and hence the norm of $\eta_{p}$ can not be $-1$. Therefore, $x^{2}-py^{2}=1$. As shown in \cite{Leg}, we can deduce from $(x+1)(x-1)=py^{2}$ that either \begin{equation}
\label{minus}
x+1=pb^{2}, \qquad x-1 =a^{2} \quad\mbox{and  } a^{2}-pb^{2}=-2,
 \end{equation}
 or
 \begin{equation}
\label{plus}
x+1=a^{2}, \qquad x-1 =pb^{2} \mbox{ and  } \quad a^{2}-pb^{2}=2.
 \end{equation}
\begin{rmk}
\label{rk}
Since $a^{2}+3b^{2}\equiv 2$ mod $4$, $a$ and $b$ must both be odd, hence so is $y=ab$. Clearly, the integers $a$ and $b$ must be coprime. For a prime $p\equiv 3$ mod $4$, $-2$ is a quadratic residue of $p$ if and only if $p\equiv 3$ mod $8$. Therefore, (\ref{minus}) holds precisely when $p\equiv 3$ mod $8$ and (\ref{plus}) holds precisely when $p\equiv 7$ mod $8$.
\end{rmk}

\begin{prop}
\label{main1}
Let $x+y\sqrt{p}$ be the fundamental unit of $\mathbb{Q}(\sqrt{p})$ for a prime $p\equiv 3$ mod $4$. Then
$a=c_{\frac{l}{2}-1}$ and $b=h_{\frac{l}{2}-1}$ in the notation above. In particular,
$$ c_{\frac{l}{2}-1}^{2}-h_{\frac{l}{2}-1}^{2}=(-1)^{\frac{p+1}{4}}2.$$
\end{prop}
\noindent Note that
\begin{equation}
\label{yl}
y=ab=h_{l-1}=h_{\frac{l}{2}-1}c_{\frac{l}{2}-1}.
\end{equation}
We first prove the following lemma.
\begin{lem}
\label{cp}
$h_{\frac{l}{2}-1}$ and $c_{\frac{l}{2}-1}$ are co-prime.
\end{lem}
\begin{proof}
Suppose $d$ is the greatest common divisor of $h_{\frac{l}{2}-1}$ and $c_{\frac{l}{2}-1}=h_{\frac{l}{2}}+h_{\frac{l}{2}-2}$. By (\ref{yl}), $d$ must be odd since $y$ is odd (see Remark \ref{rk}). We have
\begin{align}
h_{\frac{l}{2}}+h_{\frac{l}{2}-2} - a_{\frac{l}{2}}h_{\frac{l}{2}-1}& \equiv 0 \mbox{  mod  }d \nonumber \\
\Rightarrow  2h_{\frac{l}{2}-2} & \equiv 0 \mbox{  mod  }d \;\;\;\mbox{ by } (\ref{rr}).
\end{align}
It follows that $d$ must divide $h_{\frac{l}{2}-2}$. But $d$ is also a divisor of $h_{\frac{l}{2}-1}$ and  $h_{\frac{l}{2}}+h_{\frac{l}{2}-2}$. Hence $d$ divides all three consecutive convergents $h_{\frac{l}{2}}$, $h_{\frac{l}{2}-1}$ and $h_{\frac{l}{2}-2}$. From the recurrence relation (\ref{rr}), it follows that $d$ divides $h_{\frac{l}{2}-3}$, $h_{\frac{l}{2}-4}$ and so on. By going backwards, we can conclude $d$ divides $h_{2} = 1 + a_{1}a_{2}$ and $h_{1} = a_{1}$, i.e. $d=1$.
\end{proof}
{\underline{Proof of Proposition \ref{main1}}}: In view of (\ref{yl}) and Lemma \ref{cp}, it is enough to show that $q\mid a$ if and only of $q\mid c_{\frac{l}{2}-1}$. We first consider $p\equiv 3$ mod $8$. By a result of Golubeva (\cite{G1}), the period length $l$ of the RCF of $\sqrt{p}$ is of the form $4\lambda +2$.

Suppose $q$ is a prime dividing $c_{\frac{l}{2}-1}$. Then using (\ref{rr}) repeatedly,
 \begin{align*}
 h_{\frac{l}{2}} & = - h_{\frac{l}{2}-2} \mbox{ mod } q,\\
h_{\frac{l}{2}+1} & = a_{\frac{l}{2}+1}h_{\frac{l}{2}} + h_{\frac{l}{2}-1} \equiv  -a_{\frac{l}{2}-1}h_{\frac{l}{2}-2} + h_{\frac{l}{2}-1} = h_{\frac{l}{2}-3} \text{ mod } q, \\
h_{\frac{l}{2}+2}
& = a_{\frac{l}{2}+2} h_{\frac{l}{2}+1} + h_{\frac{l}{2}} \equiv a_{\frac{l}{2}-2} h_{\frac{l}{2}-3} - h_{\frac{l}{2}-2} = -h_{\frac{l}{2}-4} \text{ mod } q.
\end{align*}
Note that we are getting the positive sign on the right hand side of the congruence exactly when the index on the left hand side is even.  By continuing in this way, we find that
\begin{equation}
\label{deb2}
h_{l-2}  \equiv h_{0} =1\text{ mod } q.
\end{equation}
Suppose the prime $q$ does not divide $a$. In view of (\ref{yl}), $q$ must divide $b$. By (\ref{funda1}) and (\ref{minus}),  $$pb^{2}=x+1=nh_{l-1}+h_{l-2}+1.$$
\noindent It follows that
 \begin{equation}
 \label{deb4}
 h_{l-2}  \equiv -1\text{ mod } q.
  \end{equation}
By (\ref{deb2}) and (\ref{deb4}), $q$ must be $2$ but that is absurd as $b$ is odd. Thus, $q$ divides $a$.

Conversely, suppose $q$ is a prime dividing $a$. By (\ref{funda1}), (\ref{minus}) and (\ref{yl}),  $$a^{2}=x-1=nh_{l-1}+h_{l-2}-1=nab+h_{l-2}-1.$$
\noindent It follows that
 \begin{equation}
 \label{deb6}
 h_{l-2}  \equiv 1\text{ mod } q.
  \end{equation}
\noindent If possible, suppose $q$ does not divide $c_{\frac{l}{2}-1}$. Then $q$ must divide $h_{\frac{l}{2}-1}$ by (\ref{yl}). Using (\ref{rr}) repeatedly, we have
\begin{align*}
 h_{\frac{l}{2}} & = a_{\frac{l}{2}}h_{\frac{l}{2}-1} + h_{\frac{l}{2}-2}\equiv h_{\frac{l}{2}-2} \mbox{ mod } q,\\
h_{\frac{l}{2}+1} & = a_{\frac{l}{2}+1}h_{\frac{l}{2}} + h_{\frac{l}{2}-1} = a_{\frac{l}{2}-1}h_{\frac{l}{2}} + h_{\frac{l}{2}-1} \equiv a_{\frac{l}{2}-1}h_{\frac{l}{2}-2} - h_{\frac{l}{2}-1} = -h_{\frac{l}{2}-3} \text{ mod } q \nonumber \\
h_{\frac{l}{2}+2}
& = a_{\frac{l}{2}+2} h_{\frac{l}{2}+1} + h_{\frac{l}{2}} \equiv -a_{\frac{l}{2}-2} h_{\frac{l}{2}-3} + h_{\frac{l}{2}-2} = h_{\frac{l}{2}-4} \text{ mod } q.
\end{align*}
Note that we are getting the positive sign on the rightmost term exactly when the index on the left hand side is odd.  By continuing in this way, we find that
\begin{equation}
\label{deb0}
h_{l-2}  \equiv h_{0} =-1\text{ mod } q.
\end{equation}
By (\ref{deb6}) and (\ref{deb0}), $q$ must be $2$ which is absurd as $a$ is odd. Therefore, any prime dividing $a$ must divide $c_{\frac{l}{2}-1}$.

Since $(a,b)=1=(h_{\frac{l}{2}-1}, c_{\frac{l}{2}-1})$, we must have $a=c_{\frac{l}{2}-1}$, and hence $b=h_{\frac{l}{2}-1}$ by (\ref{yl}). The argument in the case of $p\equiv 7$ mod $8$ is identical, noting that the period length of $\sqrt{p}$ is divisible $4$ by a result of Golubeva (\cite{G1}) in this case.
\quad $\square$\\

Later we need the following corollary, the latter part of which is mentioned without proof in \cite{G2}. This result may be well-known in literature, nevertheless we include a proof that follows immediately from Proposition \ref{main1}.
\begin{cor}
\label{central}
The central term $a_{\frac{l}{2}}$ in the RCF of $\sqrt{p}$ is odd, and is either $n=\lfloor\sqrt{p}\rfloor$ or $n-1$.
\end{cor}
\begin{proof}
By Proposition \ref{main1} and Remark \ref{rk}, both  $a=h_{\frac{l}{2}}+h_{\frac{l}{2}-2}$ and $b=h_{\frac{l}{2}-1}$ are odd. Considering parity in $a= a_{\frac{l}{2}}h_{\frac{l}{2}-1}+2h_{\frac{l}{2}-2}$, we find that $a_{\frac{l}{2}}$ must be odd.

It is well known that $a_{i}\leq n$ for all $i=1,\; \ldots, \;l$. If possible, let $a_{\frac{l}{2}}<\sqrt{p}-2$. Then $p > a_{\frac{l}{2}}^{2} + 4a_{\frac{l}{2}} + 4$. We have
\begin{align*}
a^{2}-pb^{2} & < (a_{\frac{l}{2}}h_{\frac{l}{2}-1} +2h_{\frac{l}{2}-2})^{2} -(a_{\frac{l}{2}}^{2} + 4a_{\frac{l}{2}} + 4)h_{\frac{l}{2}-1}^{2}\\
  & = 4h_{\frac{l}{2}-2}^{2}- 4h_{\frac{l}{2}-1}^{2} + 4a_{\frac{l}{2}}h_{\frac{l}{2}-1}(h_{\frac{l}{2}-2}-h_{\frac{l}{2}-1})\\
& < -4 \qquad (\mbox{  since  } h_{\frac{l}{2}-2} < h_{\frac{l}{2}-1}).
 \end{align*}
But the final inequality contradicts  $a^{2}-pb^{2}=\pm 2$. Therefore, we must have $\sqrt{p} - 2 < a_{\frac{l}{2}} \leq n,$
and $a_{\frac{l}{2}}$ is either $n$ or $n-1$.
\end{proof}

\section{Families of $p$ satisfying Mordell's Conjecture}
The following theorem allows us to confirm that Mordell's conjecture holds when the RCF of $\sqrt{p}$ has period length $2$, $4$, $6$ or $8$.
\begin{thm}
\label{mord}
 Let $x+y\sqrt{p}$ denote the fundamental unit of the real quadratic field $\mathbb{Q}(\sqrt{p})$, where $p$ is a prime congruent to $3$ modulo $4$. Then $p$ divides $y$ if and only if $p$ divides  $h_{\frac{l}{2}-1}$.
\end{thm}
\begin{proof}
By Proposition \ref{main1}, we have
\begin{equation}
\label{deb5}
 c_{\frac{l}{2}-1}^{2}- ph_{\frac{l}{2}-1}^{2}=\pm 2,
 \end{equation}
where $y=h_{\frac{l}{2}-1}c_{\frac{l}{2}-1}$. Suppose $p$ divides $y$. If $p$ does not divide $h_{\frac{l}{2}-1}$, it has to divide $c_{\frac{l}{2}-1}$. By (\ref{deb5}), $p$ must divide $2$, which is absurd.
\end{proof}
\begin{cor}
\label{mord2}
Let $p$ be a prime such that the RCF of $\sqrt{p}$ has period length $2$, $4$ or $6$. Then Mordell's conjecture holds for the fundamental unit of $\mathbb{Q}(\sqrt{p})$.
\end{cor}
\begin{proof}
When $l=2$, $h_{\frac{l}{2}-1}=h_0 = 1$. By Theorem \ref{mord}, $p$ cannot divide $y$.

When $l=4$, $\sqrt{p}=\langle n; \overline{\alpha, \beta, \alpha, 2n}\rangle$ and it is well known that $\alpha< \sqrt{p}$. Therefore, $h_{\frac{l}{2}-1}=h_1 = \alpha $ cannot be divisible by $p$.

When $l=6$, we have $\sqrt{p} = \langle{n; \overline{\alpha,\beta,\gamma,\beta,\alpha,2n}\rangle}$.
We show that $h_2 \leq 2n$ in the following proposition (Proposition \ref{per6}). Granting that, we have
$$h_{\frac{l}{2}-1}=h_2 \leq 2n \leq n^{2}<p.$$
\noindent Hence $p$ cannot divide $y$ by the theorem above.
\end{proof}

\begin{prop}
\label{per6}
Let $\sqrt{p} = \langle{n; \overline{\alpha,\beta,\ldots}\rangle}$.
Then $h_2 \leq 2n$
\end{prop}
\begin{proof}
We have $p=n^{2}+t, \;\;t\leq 2n$ and
$$\sqrt{p}=n+\sqrt{p}-n=n+\frac{1}{\frac{\sqrt{p}+n}{p-n^{2}}}=n+\frac{1}{\frac{2n+(\sqrt{p}-n)}{t}},
\qquad 0<\sqrt{p}-n<1.$$
 Now, the coefficient $\alpha$ in the continued fraction of $\sqrt{p}$ is given by
 \begin{equation*}2n=t\alpha+r_{1}, \;\;\quad 0<r_{1}<t.\end{equation*}
Note that $r_{1}=0$ would imply that $\sqrt{p}=\langle{n;\overline{\alpha,2n}}\rangle$. Now,
$$\sqrt{p}= n+\frac{1}{\frac{t\alpha+(\sqrt{p}-(n-r_{1}))}{t}}
=n+\frac{1}{\alpha+\frac{1}{\frac{t(\sqrt{p}+(n-r_{1}))}{p-(n-r_{1})^{2}}}}.$$
We observe that
\begin{equation*}
\begin{split}
p-(n-r_{1})^{2}& =(p-n^{2})+2nr_{1}-r_{1}^{2}=t+(t\alpha +r_{1})r_{1}-r_{1}^{2}=t(1+\alpha r_{1}),\\
t(\sqrt{p}+(n-r_{1}))& =t(2n-r_{1}+\sqrt{p}-n).
\end{split}
\end{equation*}
It follows that the next coefficient $\beta$ in the continued fraction of $\sqrt{p}$ is given by
\begin{equation*}
2n-r_{1}=(1+\alpha r_{1})\beta+r_{2}, \;\;\;\;0\leq r_{2}<1+\alpha r_{1}.\end{equation*}
As $r_{1}\geq 1$, we conclude from the last equality that $\alpha\beta<2n$ and $h_{2}=\alpha\beta+1\leq 2n$.
\end{proof}

\begin{cor}
\label{mord8}
Let $p$ be a prime such that the RCF of $\sqrt{p}$ has period length $8$. Then Mordell's conjecture holds for the fundamental unit of $\mathbb{Q}(\sqrt{p})$.
\end{cor}
\begin{proof}
Let $p$ be a prime such that
$$\sqrt{p}=\langle{n; {\overline{\alpha, \beta, \gamma, \delta, \gamma, \beta, \alpha, 2n}}}\rangle.$$
By Corollary \ref{central}, $\delta$ is the odd number in $\{n,\;n-1\}$. Since $l=8$, we know that $p\equiv 7$ mod $8$ (\cite{G1}). By Proposition (\ref{main1}),
\begin{equation}
\label{new1}
c_{3}^{2}-ph_{3}^{2}=2.
\end{equation}By Theorem \ref{mord}, $p$ divides $y$ if and only if $p$ divides $h_{3}$. Here
$$h_{3}=\gamma h_{2}+h_{1}=\alpha\beta\gamma+\gamma+\alpha, \qquad c_{3}=h_{4}+h_{2}=\delta h_{3}+2h_{2}.$$  By Proposition \ref{per6} and the fact that $\alpha, \beta, \gamma\leq n$, we have
\begin{equation}
    \label{bd1}
h_{2}\leq 2n, \qquad h_{3}<2n^2 + 2n<3p.
\end{equation}
If $p$ divides $h_{3}$, then $h_{3}$ must be either $p$ or $2p$. If $h_{3}=2p$ it follows from (\ref{new1})  that
$c_{3}^{2}\equiv 2$ mod $4$ which is absurd. Therefore, $p$ divides $h_3$ implies
\begin{equation}
    \label{new2}
h_{3}=p, \qquad c_{3} = \delta h_3 +2h_2 = \delta p +2h_2 .
\end{equation}
When $\delta=n-1$, we have
$$c_3 =p(n-1)+2h_2 <p\sqrt{p}-(p-4n)<p\sqrt{p}-2 \mbox{ for }p\geq 23 $$ and consequently, $c_{3}^{2} < p^{3}-2$ which contradicts (\ref{new1}).  For $p< 23$, we can directly check that $p$ does not divide $h_3$.

Now we need only to rule out the case $\delta = n$ which occurs when $n$ is odd. By substituting from (\ref{new2}) in (\ref{new1}), we obtain
\begin{align}
\label{c}
(np+2h_2 )^2-p^{3} & = 2\\
\implies 4h_{2}^{2}& \equiv 2 \mbox{ mod } p \nonumber \\
\label{bd2} \implies \lambda p + 1 & = 2h_{2}^{2}\leq 8n<8p
\end{align}
From (\ref{bd2}), it follows that $\lambda$ must be an odd positive integer not exceeding $8$. Considering $2h_{2}^{2}=\lambda p +1 \equiv 7\lambda+1$ mod $8$, we can conclude that $\lambda=1$ if $h_2$ is even and $\lambda=7$ when $h_{2}$ is odd. Substituting $2h_{2}^{2}=\lambda p +1$ in (\ref{c}), we obtain
\begin{equation}
\label{cruc}
 4nh_{2} = p(p-n^{2})-2\lambda, \qquad \lambda\in\{1,\;7\}.
\end{equation}
When $p-n^{2}\geq 9$, it follows from (\ref{cruc}) that $$4nh_{2}>8p>8n^{2}, \mbox{ i.e. } h_{2}>2n$$ which is absurd by (\ref{bd1}).

It remains to consider $p-n^2 <9$. Since $n$ is odd, and $p\equiv 7$ mod $8$, we need only restrict to the case $p=n^{2}+6$. By (\ref{cruc}), $n$ must divide $p^{2}-2\lambda=(n^{2}+6)^{2}-2\lambda$ where $\lambda = 1$ or $7$. Therefore, $n$ must be an odd factor of $34$ or $22$, i.e., $p=1^2+6$, $17^2 +6$ or $11^2 +6$. We can easily check that $p$ does not divide $h_{3}$ for $p=7$ or $p=127$, and note that $17^{2}+6$ is not a prime. We can conclude that $p$ cannot divide $h_{3}$. By Theorem \ref{mord}, $p$ cannot divide $y$ when the RCF of $\sqrt{p}$ has period $l=8$.
\end{proof}

By the conjecture of P. Chowla and S. Chowla mentioned in the introduction, there exist infinitely many primes $p$ such that $\sqrt{p}$ has period $2$, $4$, $6$ or $8$. Hence by Corollary \ref{mord2}, we obtain four conjecturally infinite families of primes for which Mordell's Conjecture holds. The existence of infinitely many primes $p$ with $\sqrt{p}$ having RCF of period length $2$, $4$ and $6$ is predicted by Bouniakowsky's conjecture (\cite{Bou}) as well. Bouniakowski's conjecture states that if $f(x)$ is an irreducible polynomial in $\mathbb{Z}[x]$ with positive leading coefficient such that $gcd\{f(n)\mid n\in \mathbb{N}\}=1$, then $f(n)$ takes infinitely many prime values as $n$ runs over natural numbers. One can readily check that any prime $p$ such that $\sqrt{p}$ has RCF of period length $2$ must be of the form $p=n^{2}+2$. Similarly, one can check that any prime $p$ such that $\sqrt{p}$ has RCF of period length $4$ must be of the form $p=(n+1)^{2}-2$. While one can not characterize primes $p$ such that  $\sqrt{p}$ has RCF of period length $6$, it can be shown for primes $p=(6k+4)^{2}+4k+3=36k^2 +52 k +19$, the RCF of $\sqrt{p}$ has period length $6$. The polynomials $x^{2}+2$, $(x+1)^{2}-2$ and $36x^2 +52 x +19$ clearly satisfy the hypotheses of Bouniakowski's conjecture. Thus, three families of primes satisfying Mordell's conjecture in Corollary  \ref{mord2} are infinite by Bouniakowski's Conjecture as well.

As the conjecture of Mordell has been verified for all primes $p<10^{7}$ (\cite{Beach}), we provide a few examples of primes $p>10^{7}$ that satisfy the conjecture. \\

\begin {table}[h!]
\begin{center}
\begin{tabular}{|c|c|c|}
\hline
$p=(n+1)^{2}-2$ & $\sqrt{p}=\langle{n,\overline{1,n-1,1,2n}}\rangle$ & $\xi_{p}=x+y\sqrt{p}$ \\
\hline
\hline
$  10017223$ & $\langle{3164,\overline{1,3163,1,6328}}\rangle$ & $10017224+3165\sqrt{10017223} $\\
\hline
$  20948927$ & $\langle{4576,\overline{1,4575,1,9152}}\rangle$ & $20948928+4577\sqrt{20948927} $\\
\hline
$  21003887$ & $\langle{4582,\overline{1,4581,1,9164}}\rangle$ & $21003888+4583\sqrt{21003887} $\\
\hline
$  21022223$ & $\langle{4584,\overline{1,4583,1,9168}}\rangle$ & $21022224+4585\sqrt{21022223} $\\
\hline
\end{tabular}
\end{center}
\end{table}

\end{document}